\documentclass[11pt]{amsart}
\usepackage{amsmath}
  \usepackage{paralist}
  \usepackage{graphics} 
  \usepackage{epsfig} 
\usepackage{graphicx}  \usepackage{epstopdf}
 \usepackage[colorlinks=true]{hyperref}
\hypersetup{urlcolor=blue, citecolor=red}

\newtheorem{theorem}{Theorem}[section]

\newtheorem{lemma}[theorem]{Lemma}
\newtheorem{proposition}{Proposition}

\theoremstyle{definition}

\newtheorem{remark}{Remark}

\newcommand{\ep}{\varepsilon}

\newcommand{\rmO}{\mathrm{O}}
\newcommand{\rmo}{\mathrm{o}}

\newcommand{\rmd}{\mathrm{d}}
\newcommand{\rme}{\mathrm{e}}
\newcommand{\rmi}{\mathrm{i}}

\title[Inhomogeneities in 3-d oscillatory media]
      {Inhomogeneities in 3 dimensional oscillatory media}

\author[Gabriela Jaramillo]{}


 \email{jara0025@umn.edu}

\thanks{The author is supported by NSF grant DMS-0806614 and DMS-1311740. }

\begin{document}
\maketitle

\centerline{\scshape Gabriela Jaramillo }
\medskip
{\footnotesize
 \centerline{University of Minnesota}
   \centerline{School of Mathematics}
\centerline{127 Vincent Hall, 206 Church St SE}
   \centerline{ Minneapolis, MN 55455, USA}
} 


l

\bigskip

\begin{abstract}
We consider localized perturbations to spatially homogeneous oscillations in dimension 3 using the complex Ginzburg-Landau equation as a prototype. In particular, we will focus on inhomogeneities that locally change the phase of the oscillations. In the usual translation invariant spaces and at $ \ep=0$  the linearization about these spatially homogeneous solutions result in an operator with zero eigenvalue embedded in the essential spectrum. In contrast, we show that when considered as an operator between Kondratiev spaces, the linearization is a Fredholm operator. These spaces consist of functions with algebraical localization that increases with each derivative.  We use this result to construct solutions close to the equilibrium via the Implicit Function Theorem and derive asymptotics for wavenumbers in the far field. 
\end{abstract}

\section{Introduction}

This paper is concerned with the effects of inhomogeneities in oscillatory media. As a prototype we study the complex Ginzburg-Landau equation,
\begin{equation}\label{CGL1}
A_t = (1+ i\alpha ) \Delta A + A - ( 1 + i\gamma ) A |A|^2,
\end{equation}
which is known to approximate the phase and amplitude of modulation patterns in reaction diffusion systems near a supercritical Hopf bifurcation \cite{AK02}. Stationary in time inhomogeneities which produce a localized change in the phase of oscillations in such a system can be well modeled by the inclusion of a term $ i \ep g(x) A$ in \eqref{CGL1}. The effects of such inhomogeneities can vary dramatically depending on the sign of $\ep$ and the space dimension. This has been explored formally in the phase-diffusion approximation in \cite{SM06}, and for general reaction-diffusion equations and radially symmetric inhomogeneities in \cite{KS07}. Most notably, inhomogeneities can create wave sources in space dimension 1 and 2. In dimension 3 and radial geometry it was shown in  \cite{KS07} that sources are weak, that is, wavenumbers decay in the far field. In this note, we establish a similar result \emph{without} the assumption of radial symmetry and without relying on spatial dynamics. In addition, we relax the assumption of spatial decay of $g(x)$.

To accomplish this task we hope to use the Implicit Function Theorem to find approximations near a spatially homogeneous solution to the complex Ginzburg-Landau equation. As we will see, the linearization about these steady solutions results in an operator which is not Fredholm in the usual translation invariant Sobolev spaces. This is a consequence of zero belonging to the essential spectrum, which in some instances can be taken care of by working in exponentially localized spaces. However, since we will be considering algebraically localized inhomogeneities these spaces do not provide the appropriate framework, the use of exponential weights would turn the linearization into a semi-Fredholm operator with infinite dimensional cokernel. Instead, we will use an approach based on functional analysis and try to recover Fredholm properties of the linearization using Kondratiev spaces and the results from \cite{M79}, where it was shown that the Laplacian is a Fredholm operator. 

In addition to Kondratiev spaces, our method relies on weighted Sobolev space. We will see that for certain weights of the form  $( 1+ |x|^2)^{\delta/2}$ the linearization about steady solutions possesses a cokernel. We will therefore consider an Ansatz which adds far field corrections and obtain as a result an invertible operator. This approach works well for weights with $\delta<1/2$, however for $\delta>1/2$ these correction terms prove to be problematic since they result in nonlinearities which are not well defined, i.e. they do no belong to the correct weighted space. The same is true in the 2 dimensional case for all weights that account for decaying inhomogeneities. we hope to address these issues in the future  and restrict ourselves in the present paper to the 3 dimensional case with $\delta<1/2$. This will provide a straight forward example where the advantage of viewing the linearization in the setting of  Kondratiev spaces can be appreciated without the extra complications coming form the nonlinearity.

We begin the analysis by considering the spatially homogeneous solution $A_*(t)=\rme^{-\rmi \gamma t}$ of equation \eqref{CGL1} and looking for approximations of the form $A(x,t)=(1-s(x))\rme^{-\rmi (\gamma t - \phi(x))}$. In Section \ref{Proof} we will show, using Lyapunov-Schimdt reduction, that in dimension 3 it is possible to find solutions near $A_*$. The asymptotics for the function $\phi(x)$ will show that in the far field the wavenumber $k \sim \nabla \phi$ decays to zero and hence target patterns will not form. We state this result in the following Theorem:

\begin{theorem}\label{Th:main}
Suppose  $\delta \in (-1/2,1/2)$, $g \in L^2_{\delta+2}$, and $1 + \alpha \gamma>0$. Then, there exist $\ep_0>0$ and smooth functions $S(x,\ep)$ and $\Phi(x,t; \ep)$ such that $$A(x,t; \ep) =  S(x, \ep)  \rme^{  \Phi(x,t;\ep)}$$ is a family of solutions to  \eqref{CGL1} near $A= \rme^{-\rmi \gamma t}$  for all $\ep \in (-\ep_0,\ep_0)$. Furthermore, for fixed $\ep \in (-\ep_0, \ep_0)$ and $t$,  the functions $S(x;\ep)$ and $\Phi(x,t;\ep)$ satisfy the following asymptotics in $x$,
\begin{align*}
|S(x,\ep)-1| \leq & C |x|^{-(\delta + 2.5)},\\
\Phi(x,t;\ep) =&  -\rmi \gamma t + \rmi \frac{c(\ep)}{|x|}\left( 1 + \rmo_1( 1/|x|) \right),
\end{align*}
as $|x| \rightarrow \infty$, where $c(\ep)$ is a smooth function satisfying the expansion $c(\ep)  = \ep c_1 + O(\ep^2)$. In particular, $$c_1 = \dfrac{1}{4 \pi ( 1 + \alpha \gamma) } \int g \;\rmd x.$$

\end{theorem}

\begin{remark}

\begin{enumerate}
\item Notice that we do not have asymptotic predictions for the amplitude, just an upper bound on the rate of its decay. 

\item The values of $\delta$ are related to the choice of spaces we make. In the case of $\delta \in (-1/2,1/2)$ our analysis shows that the linearization about the steady solution $A=\rme^{-\rmi \gamma}$ is a Fredholm operator of index $i=-1$. If we consider weights with $\delta \in (1/2+m, 1/2+2m)$, for $m \in \mathbb{N}$, the linearization is again a Fredholm operator, but now with a larger co-kernel consisting of harmonic polynomials of degree $m-1$. In this case, it seems reasonable to add to the Ansatz a series of correction terms which would span the cokernel of our linearization. In particular, these terms should consist of derivatives of the fundamental solution $\frac{1}{|x|}$ of all degrees up to $m-1$. The difficulty in this case is that this type of Ansatz results in a non-linear operator which is not well defined in $L^2_{\delta+2}$ (see Proposition \ref{prop:Nonlinearity}). Nonetheless, because $L^2_{\alpha} \subset L^2_{\beta} $ for  $\beta< \alpha$, if we consider a very localized inhomogeneity we can always assume it is in a space $L^2_{\delta +2}$ with $-1/2<\delta< 1/2$. In other words, Theorem \ref{Th:main} holds for $g \in L^2_{\sigma}$ with $\sigma>3/2$, and in this case we take $\delta = 1/2$ for the bounds of $|S(x,\ep) -1|$. However, for these values of $\sigma$ it is still an open problem to determine if this bound is sharp.

\item In the case of $\delta \in (-3/2,-1/2)$, we can consider spaces which yield an invertible linearization. Our analysis then shows that the amplitude $S(x,\ep)$ should obey the same decay as stated in Theorem \ref{Th:main}, but we do not expect phase decay at order O(1/|x|). In fact, the coefficient of the leading order term, $\int g \rmd x$, is not necessarily defined when $g$ is in $L^2_{\delta+2}, \delta<-1/2$. Our result would only give decay associated with the function space $M^{2,2}_{\delta}$ (see Lemma 3.2).

\item Finally, we just point out that we are not interested in studying inhomogeneities with slow decay, $g\sim |x|^{-\alpha} \alpha<1$, or that grow algebraically, and so we do not look at the case when $\delta<-3/2$.
\end{enumerate}
\end{remark}

The predictions of Theorem \ref{Th:main}  agree with the results found in \cite{KS07}, where the authors show that in the more general case of reaction diffusion equations and in dimensions 3 and higher, there exists only contact defects (the wave number $k  \sim \nabla \phi \rightarrow 0$ in the far field) and obtain asymptotics for the wavenumber $k$, 
\[ k(r,\ep) = \frac{M \ep}{r^{n-1}} ( \hat{c} + \rmO_{1/r}(1) ), \]
where the notation $\rmO_{y}(1)$ means that these terms go to zero as $y \rightarrow 0$. This implies that for large values of $|x|$ and fixed $\ep$ we do not see a pace maker effect. Nonetheless, if we fix $|x|$ large we can approximate the group velocity, $c_g$, for the family of solutions $A(x,t;\ep)$ in terms of $\ep$:
\[ c_g(\ep) = 2(\alpha - \gamma)k  \sim -2  (\alpha - \gamma) \frac{\ep c_1}{|x|^2}.\]

\begin{figure}
\begin{center}
\includegraphics[scale=0.5]{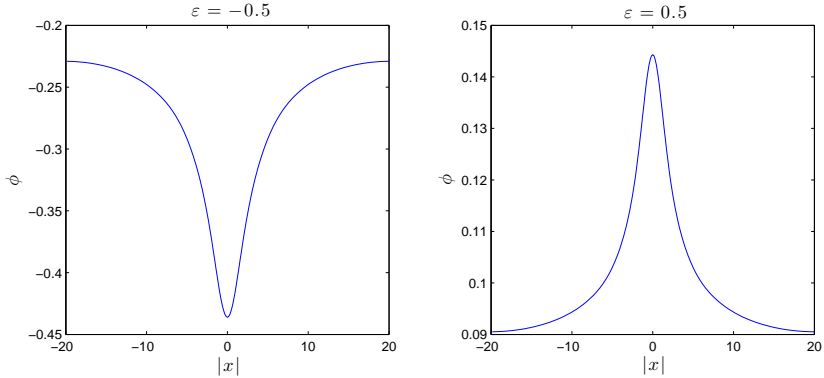} 
\end{center}
\caption{Plot of the phase vs. $x$-axis for the cross section $y=z=0$. For the parameter values used in the simulation the expression $(\alpha-\gamma)<0$. As a result, a negative phase gradient as $|x| \rightarrow \infty$ indicates a positive group velocity, whereas a positive phase gradient as $|x| \rightarrow \infty$ indicates a negative group velocity.}
\label{groupvelocity}
\end{figure}
In particular, if $\ep (\gamma- \alpha) \int g  >0 $ then $c_g>0$ and we obtain weak wave sources. These results were confirmed in numerical simulations with a cubic domain of length $l=40$, parameter values $\alpha =1, \gamma=5$, and with the following inhomogeneity
\[
g(x,y,z) =  \dfrac{1}{( 1 + 1/4(x-10)^2+ 2(y-10)^2 +(z-10)^2)^{3.2/2}  },
\]
  (see figure \ref{groupvelocity}). All simulations were done with an exponential time differencing algorithm (ETDRK4) following the methods found in \cite{K03,KT05}.

This paper is organized as follows:  In Section \ref{weighted}, we define weighted Sobolev spaces and Kondratiev spaces and state Fredholm properties for the Laplace operator. Next, in Section \ref{Proof} we give a proof of our main result and finally, in Section \ref{numerics}, we present numerical simulations of our results. In particular, we show the decay rates for the amplitude and phase agree with our predictions.

\section{Weighted and Kondratiev spaces} \label{weighted}
\subsection{Weighted spaces}
In this paper we consider the weight $ \langle x \rangle = (1 + |x|^2)^{1/2}$ and define the weighted Sobolev spaces, $W^{k,p}_{\delta}$, as the completion of $C^{\infty}_0(\mathbb{R}^n)$ under the norm
\[ \|u \|_{W^{k,p}_{\delta}} = \left( \sum_{|\alpha|\leq k} \| D^{\alpha} u \cdot \langle x \rangle^{\delta} \|^p_{L^p} \right)^{1/p},\]
with $ 1< p < \infty$, $\delta \in \mathbb{R}$ and $k \in \mathbb{N}$. Notice that we have inclusions of the form $W_{\beta}^{k.p}\subset W^{k,p}_{\alpha}$ for any real numbers $\alpha, \beta$ such that $ \alpha < \beta$. Furthermore, we have the following proposition which was proven in \cite{JS13}.

\begin{proposition}\label{Sobolev}
The operator $\Delta -a : W_{\delta}^{2,p} \rightarrow L^p_{\delta}$ is invertible for all real numbers $a >0$ and $p \in (1, \infty)$. 
\end{proposition}

The above proposition also shows why  the Laplace operator does not have closed range when considered in the setting of weighted Sobolev spaces: just as in the case of $\Delta: H^2 \rightarrow L^2$, we can construct Weyl's sequences for the Laplace operator proving that zero is in the essential spectrum. We summarize this results as a lemma:
\begin{lemma}\label{notFredholm}
 The operator $\Delta_{\delta}:W^{2,p}_{\delta} \rightarrow L^p_{\delta}$ is not a Fredholm operator for $p \in (1, \infty)$.
\end{lemma}

\subsection{Kondratiev spaces}

A slight variation of the above spaces are Kondratiev spaces, where the exponent in the weight $\langle x \rangle$ is increased by one every time we take a derivate. We denote them here by $M^{k,p}_{\delta}$, and defined them as the completion of $C^{\infty}_0( \mathbb{R}^n)$ under the norm
\[ \|u \|_{M^{k,p}_{\delta}} = \left( \sum_{|\alpha|\leq k} \| D^{\alpha} u \cdot \langle x \rangle^{\delta+|\alpha|} \|^p_{L^p} \right)^{1/p}.\]
Again we let $1<p<\infty$, $\delta \in \mathbb{R}$, and $k \in \mathbb{N}$.

In general, Kondratiev spaces are studied in connection with boundary value problems for elliptic equations in domains with critical points \cite{Kon67}. They also appear in the setting of unbounded domains. For example, Nirenberg and Walker showed in \cite{NW73} that a class of elliptic operators with coefficients that decay sufficiently fast at infinity have finite dimensional kernel. Additionally,  McOwen and Lockhart used this spaces to study Fredholm properties of elliptic operators and systems of elliptic operators in non-compact manifolds \cite{L81, LM83, LM85}. Moreover, Kondratiev spaces have also been used in the description of far field asymptotics for fluid problems, in particular when studying the flow past obstacles, since they lend themselves to the study of problems in exterior domains (see \cite{SNW86} for the case of $\mathbb{R}^3$ and \cite{MSU08} for an application towards bifurcation theory). More recently, a variant of these spaces was used in \cite{MR13} to study Poisson's equation in a one-periodic infinite strip $Z= [0,1] \times \mathbb{R}$. 
 
The main advantage for us is that in Kondratiev spaces the Laplace operator is a Fredholm operator. These results are shown in McOwen's paper \cite{M79} and are summarized in the following theorem.
\begin{theorem}\label{McOwen}
Let $1<p=\frac{q}{q-1}<\infty$,  $n \geq 2$, and $\delta \neq -2 + n/q + m $ or $\delta \neq -n/p -m $, for some $m \in \mathbb{N}$.  Then 
\[ 
\Delta: M^{2,p}_{\delta} \rightarrow L^p_{\delta+2}, 
\]
is a Fredholm operator and
\begin{enumerate}
\item for $-n/p < \delta < -2 + n/q$ the map is an isomorphism;
\item for $-2 + n/q + m < \delta< -2 + n/q + m+1$ , $m \in \mathbb{N}$, the map is injective with closed range equal to 
\[
R_m = \left\{ f \in L^p_{\delta +2} : \int f(y)H(y) =0 \  \text{for all } \ H \in \bigcup_{j=0}^m \mathcal{H}_j\right\}; 
 \]
\item for $-n/p - m -1 < \delta <-n/p -m $, $m \in \mathbb{N}$, the map is surjective with kernel equal to 
\[ N_m = \bigcup_{j=0}^m \mathcal{H}_j
.\]
\end{enumerate}
Here,  $\mathcal{H}_j $ denote the harmonic homogeneous polynomials of degree $j$.

On the other hand, if $\delta = -n/p - m $ or $\delta = -2 + n/q + m$ for some $m \in \mathbb{N}$, then $\Delta$ does not have closed range.
 \end{theorem}

\section{ Proof of Theorem \ref{Th:main}}\label{Proof}
To facilitate the analysis we will split this section into four parts. In Subection \ref{setup} we describe how we set up the problem and how we obtain a linearization which is easier to work with. Next, in Subsection \ref{proof} we state conditions that allow us to use the Implicit Function Theorem and derive expansions for the amplitude and phase, effectively proving the results of Theorem \ref{Th:main}. Finally, in the last two subsections we show that the linearization is invertible and the nonlinear operator associated to our problem is well defined.

\subsection{Set up}\label{setup}

We recall here our main equation, the complex Ginzburg-Landau equation in dimension 3,
\begin{equation}\label{CGL}
 A_t = (1+ i \alpha) \Delta A + A - ( 1 + i \gamma) A |A|^2 + i \ep g(x) A,\end{equation} 
where $g(x)$ is a localized real valued function and $\ep$ is small. In what follows we describe how we arrive at our linearization.

 We pass to a corotating frame $A = e^{- i \Omega t} \tilde{A},$  so that $\tilde{A}$ satisfies the following equation,
\begin{equation}\label{modCGL}
 \tilde{A}_t = (1+ i \alpha) \Delta \tilde{A} +( 1+ i \Omega)  \tilde{A} - ( 1 + i \gamma) \tilde{A} |\tilde{A}|^2 + i \ep g(x) \tilde{A}.\end{equation}

At parameter values $\Omega=\gamma$ and $\ep =0$, the function $\tilde{A}_*=1$ is a solution to \eqref{modCGL} and the linearization about this constant solution is given by the following operator, $T$:
\[ T \begin{bmatrix} s \\ \phi \end{bmatrix}= \begin{bmatrix} \Delta  -2 & - \alpha \Delta  \\ 
					\alpha \Delta  - 2 \gamma  & \Delta  \end{bmatrix} \begin{bmatrix} s \\ \phi \end{bmatrix}. \]
In Fourier space  $T$ can be represented by a matrix, $\mathcal{F}(T)(k)$, which at $k=0$ has eigenvalues $\lambda_1 = -2,$ and $\lambda_2 =0 $. This suggest that in order to simplify future computations we use the following change of coordinates, 
\[ \hat{s} = \gamma s, \quad \hat{\phi} = - \gamma s + \phi,\]
so as to diagonalize $\mathcal{F}(T)(0)$. The resulting operator that comes from the right hand side of the equations for $\hat{s}_t$ and $\hat{\phi}_t$, and which we label as $F: \mathcal{X} \times \mathbb{R} \rightarrow \mathcal{Y}$, is given by the following two components,
\begin{equation}\label{amplitude}
\begin{split}
F_1 (\hat{s},\hat{ \phi}) = &  ( 1- \alpha \gamma) \Delta \hat{s} -2 \hat{s}  - \gamma \alpha \Delta \hat{\phi}   - ( \gamma +\hat{s}) [ | \nabla \hat{s}|^2 +2 \nabla \hat{s} \cdot \nabla \hat{\phi} + | \nabla \hat{\phi}|^2] \\
&-2 \alpha | \nabla \hat{s}|^2-2 \alpha \nabla \hat{s} \cdot \nabla \hat{\phi} - \alpha \hat{ s} ( \Delta \hat{s} + \Delta \hat{\phi}) - \frac{3}{\gamma} \hat{s}^2 - \frac{1}{\gamma^2}\hat{s}^3,
\end{split}
\end{equation}
\begin{equation}\label{phase}
\begin{split}
 F_2(\hat{s}, \hat{\phi}) = & \left( \frac{\alpha}{\gamma} + \alpha \gamma \right) \Delta \hat{s} + ( 1 + \alpha \gamma) \Delta\hat{ \phi} + \alpha \hat{s} ( \Delta \hat{ s} + \Delta \hat{\phi}) + 2 \alpha \nabla \hat{s} \cdot \nabla \hat{\phi} \\
 &+( \gamma - \alpha + \hat{s}) \left[ | \nabla \hat{s}|^2 + 2 \nabla \hat{s} \cdot \nabla \hat{\phi} + | \nabla \hat{ \phi} |^2 \right] + 2 \alpha | \nabla \hat{s}|^2 + \frac{3 \hat{s}^2}{\gamma} + \frac{\hat{s}^3}{\gamma^2}  \\
 &+\frac{1}{(\gamma + \hat{s})} \left[  2 | \nabla \hat{s}|^2 + 2 \nabla \hat{s} \cdot \nabla \hat{\phi} - \hat{s}^2 - \frac{\hat{s}^3}{\gamma} - \frac{ \alpha}{\gamma} \hat{s} \Delta \hat{s}\right] + \ep g(x). 
 \end{split}
\end{equation}
We now introduce the following Ansatz for equation \eqref{CGL}
\begin{equation}\label{Ansatz} 
\tilde{A}(x,t, \ep) = S(x,\ep) \rme^{ \Phi(x,t,\ep) },
 \end{equation}
 where 
\begin{align*}
S(x,\ep) &= 1 + s(x,\ep), \\
  \Phi(x, \ep) &= - \rmi ( \gamma t - \phi(x,\ep) ),\quad \mbox{with} \quad     \phi(x,\ep) = \tilde{\phi}(x,\ep) + c(\ep) \underbrace{ \dfrac{\chi(|x|)}{|x|}}_P,
 \end{align*}
and  $\chi \in C^{\infty}(\mathbb{R})$ is a cut-off function equal to zero near the origin and equal to 1, for  $|x|>2$. This amounts to letting $\hat{\phi} = \tilde{\phi} + c P(x)$ in \eqref{amplitude} and \eqref{phase}, and results in a nonlinear operator which we again label as $F:\mathcal{X} \times \mathbb{R}^2 \rightarrow \mathcal{Y}$. In the last section we show that there exists spaces $\mathcal{X} $ and $\mathcal{Y}$ such that $F$ is well defined and smooth. We will also look at the properties of its linearization, $L: \mathcal{X} \times \mathbb{R} \rightarrow \mathcal{Y}$, in  Subsection  \ref{linear}, but we explicitly write the form of this linear operator for future reference here
  
  \begin{equation*}
L \begin{bmatrix} \hat{s} \\ \tilde{\phi} \\ c \end{bmatrix} = \begin{bmatrix} (1- \alpha \gamma) \Delta -2 & - \alpha \gamma \Delta & -\alpha \gamma \Delta P \\
																( \alpha \gamma + \frac{\alpha}{\gamma}) \Delta  &  ( 1 + \alpha \gamma) \Delta & ( 1 + \alpha \gamma) \Delta P \end{bmatrix}  \begin{bmatrix} \hat{s} \\ \tilde{\phi} \\ c \end{bmatrix}. \end{equation*}

We also clarify that in the rest of the paper we will write $s$ instead of  $\hat{s}$.

  \subsection{Main results: Expansions for phase $\phi$ and amplitude $s$}\label{proof}
For the remainder of the paper we let $\mathcal{X} = W^{2,2}_{\delta+2}\times M^{2,2}_{\delta}$ and $\mathcal{Y}=L^2_{\delta+2} \times L^2_{\delta+2}$.  The next proposition, together with the Implicit Function Theorem, show the existence of solutions to \eqref{CGL}. 
 
  \begin{proposition}\label{invertible}
Let $\delta \in (-1/2,1/2)$ and let $g \in L^2_{\delta +2}$. Then the operator $F: W^{2,2}_{\delta+2} \times M^{2,2}_{\delta} \times \mathbb{R}^2 \rightarrow L^2_{\delta+2} \times L^2_{\delta+2}$ defined by \eqref{amplitude} and \eqref{phase} and the Ansatz \eqref{Ansatz} is smooth and its Fr\'echet derivative $DF$ evaluated at $(s, \tilde{\phi}, c; \ep) = 0$, is invertible.
\end{proposition}
We leave the proof of this result for Subsection \ref{linear} and justify the expansions and decay rates of $S(x,\ep)$ and $\Phi(x,t,\ep)$ stated in Theorem \ref{Th:main}. First, the decay rates follow from our choice of weighted spaces and the following two lemmas.

\begin{lemma}\label{decay}
Let $\gamma>-3/2$. If $f \in M^{2,2}_{\gamma}$, then $|f(x)| \leq  {C}  \langle {\bf x} \rangle^{-\gamma-3/2}$ as $| {\bf x} |\rightarrow \infty$.
\end{lemma}
\begin{proof}
Since we define the space $M^{2,2}_{\gamma}$ as the completion of $C^{\infty}_0$ under the norm $\|\cdot\|_{M^{2,2}_{\gamma}}$, it suffices to show the result for $f \in C^{\infty}_0$. Using polar coordinates we find that in dimension 3,
\begin{align*}
\int |f(\theta, R)|^2 \; \rmd \theta &= \int \left(  \int^R_{\infty} |f_r(\theta, s)| \; \rmd s \right)^2 \; \rmd \theta \\ 
&= \int \left( \int_{\infty}^R s^{-(\gamma+2)} |f_r(\theta,s)| s^{\gamma+1} s \;\rmd s \right)^2 \; \rmd \theta\\
& \leq \int \left ( \int^R_{\infty} s^{-2(\gamma+2)} \; \rmd s\right) \left( \int_{\infty}^R s^{2(\gamma+1)}|f_r(\theta,s)|^2 s^2 \; \rmd s \right) \; \rmd \theta\\
&\leq R^{-2(\gamma+2)+1} \| f_r \|_{L^2_{\gamma+1}}.
\end{align*}
Therefore $\|f(\cdot,R)\|_{L^2} \leq C R^{-\gamma -3/2}$. Similarly,
\begin{align*}
\int |f_{\theta} (\theta, R)|^2 \; \rmd \theta &= \int \left(  \int^R_{\infty} |f_{\theta r}(\theta, s)| \; \rmd s \right)^2 \; \rmd \theta\\
& = \int \left( \int_{\infty}^R s^{-(\gamma+3)} |f_{\theta r}(\theta,s)| s^{\gamma+2} s \;\rmd s \right)^2 \; \rmd \theta\\
& \leq \int \left ( \int^R_{\infty} s^{-2(\gamma+3)} \; \rmd s\right) \left( \int_{\infty}^R s^{2(\gamma+2)}|f_{\theta r} (\theta,s)|^2 s^2 \; \rmd s \right) \; \rmd \theta\\
&\leq R^{-2(\gamma+3)+1} \| f_{\theta r} \|_{L^2_{\gamma+2}}.
\end{align*}
Combining these results and using the interpolation inequality from \cite[Thm 5.9]{Adams},
\[ \|f(\cdot, R)\|_{\infty}^2 \leq \| f(\cdot,R)\|_{L^2} \| f(\cdot, R)\|_{H^1},\]
shows the result of the claim.
\end{proof}
The next lemma can be proven in a similar manner.
\begin{lemma}\label{decay2}
Let $\gamma >-1/2$. If $f \in W^{2,2}_{\gamma}$, then $|f(x)| \leq  {C} \langle {\bf x} \rangle^{-\gamma-1/2}$ as $| {\bf x} |\rightarrow \infty$.
\end{lemma}

Next, to show the expansion for the function $c(\ep) = \ep c_1+ \rmO(\ep^2)$ stated in Theorem \ref{Th:main} we use Lyapunov-Schmidt reduction and the results of the next subsection, where we show that the vector $ (0,1)^T$, spans the cokernel of the operator $\hat{L}: W^{2,2}_{\delta+2} \times M^{2,2}_{\delta} \rightarrow L^2_{\delta+2} \times L^2_{\delta+2}$ defined by the first two columns of $L$.  If we assume expansions of the form $(s,\tilde{\phi},c)(x;\ep) = \ep (s_1,\tilde{\phi}_1,c_1) + \rmO(\ep^2)$, we can obtain at order $\rmO(\ep)$ an expression for the coefficient $c_1$:
\begin{align*}
- \int g\; \rmd x =& \int ( \alpha \gamma + \frac{\alpha}{\gamma}) \Delta s_1  +  ( 1 + \alpha \gamma) \Delta \tilde{\phi}_1 + c_1( 1 + \alpha \gamma) \Delta P \; \rmd x \\
 = & - 4 \pi  ( 1 + \alpha \gamma) c_1\\
 c_1 = &  \dfrac{ \int g\; \rmd x }{ 4 \pi  ( 1 + \alpha \gamma)},
 \end{align*}
where the last two equalities follow from Theorem \ref{McOwen} and the fact that
\[ \int  \Delta \left(  \frac{\chi(|x|)}{|x|} \right)  dx = -4\pi.\]

\subsection{The Linear operator}\label{linear}
In this subsection we prove Proposition \ref{invertible}, by decomposing the linear operator $L$ as $L = [\hat{L},M]$. First, we  use the results from Section \ref{weighted}  to show that  the operator,  $\hat{L}: W^{2,2}_{\delta+2}\times M^{2,2}_{\delta} \rightarrow L^2_{\delta+2} \times L^2_{\delta+2}$, defined below, is  Fredholm with index $-1$. Next, we show that the Ansatz \eqref{Ansatz} adds good far field corrections so that the linearization, $L: W^{2,2}_{\delta+2}\times M^{2,2}_{\delta}\times \mathbb{R} \rightarrow L^2_{\delta+2} \times L^2_{\delta+2}$ is an invertible operator. We define $\hat{L}$ explicitly for future reference:
\begin{equation}\label{linearization}
 \hat{L} \begin{bmatrix} s \\ \phi \end{bmatrix}= \begin{bmatrix} (1- \alpha \gamma) \Delta  -2 & - \gamma \alpha \Delta  \\ 
														\left (\gamma \alpha + \frac{\alpha}{\gamma} \right) \Delta   & ( 1 + \gamma \alpha) \Delta  \end{bmatrix} \begin{bmatrix}s  \\  \phi \end{bmatrix}. \end{equation}

 \begin{lemma}\label{lem:Fredholm}
 Let $\delta \in (-1/2,1/2)$,  and $1 +\gamma \alpha>0$. Then the linear operator $\hat{L} : W^{2,2}_{\delta+2} \times M^{2,2}_{\delta} \rightarrow L^2_{\delta+2} \times L^2_{\delta+2}$, defined by \eqref{linearization} is a Fredholm operator with index $i =-1$ and cokernel spanned by the vector  $(0,1)^T$.
 
 \end{lemma}

\begin{proof}
Assume 
\begin{equation}\label{Leq}
\begin{bmatrix} s\\ \phi \end{bmatrix} = \begin{bmatrix} f\\g \end{bmatrix}.
\end{equation}
 From the second component of $L$ we obtain and equation for the variable $\phi$,
\begin{equation} \label{phieq}
 \Delta \phi = \frac{g}{1 + \alpha \gamma} - \frac{ \alpha \gamma + \alpha / \gamma}{ 1 + \alpha \gamma} \Delta s.
 \end{equation}
Since $1+ \alpha \gamma >0$, we can insert the above expression for $\Delta \phi$ into the first line of equation \eqref{Leq} and solve for $s$:
\[ s = [(1 + \alpha^2)\Delta -2(1+\alpha \gamma)]^{-1} (1 + \alpha \gamma)f + [(1 + \alpha^2) \Delta -2(1+\alpha \gamma)]^{-1} \alpha \gamma g.\]
Next, we use the above result in \eqref{phieq} and obtain the following equation for $\phi$:
\[ \Delta \phi =  [ (1 + \alpha^2)\Delta -2(1+\alpha \gamma)]^{-1} [ ( 1- \alpha \gamma) \Delta  -2 ] g + \Delta  [(1 + \alpha^2) \Delta -2(1+\alpha \gamma)]^{-1} (1 + \alpha \gamma) f.\]
Our goal is to show that the right hand side is in the range of $\Delta: M^{2,2}_{\delta} \rightarrow L^2_{\delta+2}$. It is clear that the term 
\[  \Delta  [(1 + \alpha^2) \Delta -2(1+\alpha \gamma)]^{-1} (1 + \alpha \gamma) f, \]
 satisfies this requirement for any $f \in L^2_{\delta+2}$, given that it involves the Laplacian and that the operator $ [(1 + \alpha^2)\Delta -2(1+\alpha \gamma) ]^{-1} : L^2_{\delta +2} \rightarrow W^{2,2}_{\delta+2}$ is bounded. 

The results from Theorem \ref{McOwen} and our assumption that $\delta \in ( -1/2,1/2)$ require us to show that if  $g$ has average zero, then the term
\[(1 + \alpha^2)\Delta -2(1+\alpha \gamma)]^{-1} [ ( 1- \alpha \gamma) \Delta  -2 ] g \]
 also has average zero. The result follows since the operator, $A : L^2_{\delta + 2} \rightarrow L^2_{\delta +2}$ defined by $$A = [ (1 + \alpha^2)\Delta -2(1+\alpha \gamma)]^{-1} [ ( 1- \alpha \gamma) \Delta  -2 ]$$ preserves this condition. To see this, notice that the condition $\int g=0$ is equivalent to $\hat{g}(0) =0$, where $\hat{g}$ denotes the Fourier transform of $g$. Moreover, since the Fourier symbol of $A$ is given by
\[ \hat{A}(k) = \frac{(1- \alpha \gamma )|k|^2 + 2}{ ( 1+ \alpha^2) |k|^2 +2(1 + \alpha \gamma)},\]
and $1 + \alpha \gamma >0$, then  $\mathcal{F} ( A g)(0) =0$ if and only if $g(0)=0$. This proves the Lemma.

\end{proof}
\begin{remark}
Observe that the condition $1 + \alpha \gamma>0$ is also required for spectral stability, an indication that these methods are consistent with previous results.
\end{remark}
\begin{remark}
If $\delta \in (-3/2,-1/2)$ the Laplace operator is invertible.  A similar argument as in Lemma \ref{lem:Fredholm} then shows that for these values of $\delta$ the operator $\hat{L}: W^{2,2}_{\delta+2} \times M^{2,2}_{\delta} \rightarrow L^2_{\delta+2} \times L^2_{\delta+2}$ is invertible.
\end{remark}

Next, consider the  Ansatz: 
\[ \phi = \tilde{\phi} + c \underbrace{\frac{\chi(|x|)}{|x|}}_{P}, \]
where $\chi \in C^{\infty}(\mathbb{R})$ is defined as in the introduction. With this Ansatz,  the linearization of $F:W^{2,2}_{\delta+2} \times M^{2,2}_{\delta} \times \mathbb{R} \rightarrow L^2_{\delta+2} \times L^2_{\delta+2}$ about the origin is given by the operator, $L: W^{2,2}_{\delta +2 } \times M^{2,2}_{\delta} \times \mathbb{R} \rightarrow L^2_{\delta+2} \times L^2_{\delta+2}$,
\begin{equation}\label{modlinearization}
L \begin{bmatrix} s \\ \tilde{\phi} \\ c \end{bmatrix} = \begin{bmatrix} (1- \alpha \gamma) \Delta -2 & - \alpha \gamma \Delta & -\alpha \gamma \Delta P \\
																( \alpha \gamma + \frac{\alpha}{\gamma}) \Delta  &  ( 1 + \alpha \gamma) \Delta & ( 1 + \alpha \gamma) \Delta P \end{bmatrix}  \begin{bmatrix} s \\ \tilde{\phi} \\ c \end{bmatrix} ,
																\end{equation}
which we decompose as,
\[ L = \begin{bmatrix} \hat{L} & M \end{bmatrix}.\]
Here, $\hat{L}$ is the same as \eqref{linearization} and $M: \mathbb{R} \rightarrow L^2_{\delta +2} \times L^2_{\delta +2} $ is given by
\[ M c = \begin{bmatrix}  -\alpha \gamma \Delta P \\  ( 1 + \alpha \gamma) \Delta P \end{bmatrix} c.\]
It is clear that the operator $M$ is well defined since $ \Delta P = \Delta \left( \dfrac{\chi(|x|)}{ |x|} \right) $ has compact support. Notice as well that
\[ \int_{\mathbb{R}^3}  \Delta \left(  \frac{\chi(|x|)}{|x|} \right)  dx = -4\pi,\]
so that the range of $M$ and the cokernel of $L$ intersect. The Bordering lemma for Fredholm operators then shows that for $\delta \in (-1/2,1/2)$, the operator  $L: W^{2,2}_{\delta +2 } \times M^{2,2}_{\delta } \times \mathbb{R} \rightarrow L^2_{\delta+2} \times L^2_{\delta +2}$ is invertible. This proves the following result.
\begin{lemma}
Let $\delta \in (-1/2,1/2)$ and $1+ \alpha \gamma >0$. Then the operator $L: W^{2,2}_{\delta +2 } \times M^{2,2}_{\delta} \times \mathbb{R}\rightarrow L^2_{\delta+2} \times L^2_{\delta+2}$, defined by \eqref{modlinearization} is an invertible operator.
\end{lemma}

In order to finish the proof of Proposition \ref{invertible} we just need to show that the full operator $F: W^{2,2}_{\delta +2 } \times M^{2,2}_{\delta } \times \mathbb{R}^2 \rightarrow L^2_{\delta+2} \times L^2_{\delta +2}$ is well defined and smooth, justifying our assertion  that $DF(0,0,0;0)=L$. This will be done in the following section.

\subsection{Nonlinear terms}

We now consider the full non-linear operator $F: M^{2,2}_{\delta} \times W^{2,2}_{\delta+2} \times \mathbb{R}^2 \rightarrow L^2_{\delta+2} \times L^2_{\delta+2}$, given by
\begin{equation*}
\begin{split}
F_1 (s, \phi, c) = & ( 1- \alpha \gamma) \Delta s -2 s  - \gamma \alpha \Delta \phi - ( \gamma +s) [ | \nabla s|^2 +2 \nabla s \cdot \nabla \phi + | \nabla \phi|^2] \\
&-2 \alpha | \nabla s|^2 -2 \alpha \nabla s \cdot \nabla \phi  - \alpha s ( \Delta s + \Delta \phi) - \frac{3}{\gamma} s^2 - \frac{1}{\gamma^2}s^3,
\end{split}
\end{equation*}
\begin{equation*}
\begin{split}
 F_2(s, \phi, c) = & \left( \frac{\alpha}{\gamma} + \alpha \gamma \right) \Delta s + ( 1 + \alpha \gamma) \Delta \phi + \alpha s ( \Delta s + \Delta \phi)  + 2 \alpha \nabla s \cdot \nabla \phi \\
 &+( \gamma - \alpha + s) \left[ | \nabla s|^2 + 2 \nabla s \cdot \nabla \phi + | \nabla \phi|^2 \right] + 2 \alpha | \nabla s|^2  + \frac{3 s^2}{\gamma} + \frac{s^3}{\gamma^2} \\
 & +\frac{1}{(\gamma + s)} \left[  2 | \nabla s|^2 + 2 \nabla s \cdot \nabla \phi - s^2 - \frac{s^3}{\gamma} - \frac{ \alpha}{\gamma} s \Delta s\right] + \ep g(x).
 \end{split}
 \end{equation*}

We omitted the ``hats" for ease of notation and use  $\phi = \tilde{\phi} + c P $, with $P =  \dfrac{\chi(|x|)}{|x|}$.
With the help of the next  lemma we show that $F$ is well defined in the sense that all non-linear terms are in the space $L^p_{\delta +2}$. 

\begin{lemma}\label{Holder}
Let $\delta \in \mathbb{R}$. If $f,g \in W^{1,2}_{\delta +1}$, then the product $fg \in L^2_{\delta+2}$. 
\end{lemma}

\begin{proof}
This lemma is a consequence of H\"older's inequality and the Sobolev embeddings. 
\end{proof}

Notice also that if $\delta >-2$, then $W^{2,p}_{\delta+2} \subset W^{2,p}$. Furthermore, if $p=2$ we have $W^{2,2}_{\delta+2} \subset W^{2,2} \hookrightarrow BC(\mathbb{R}^3)$.

\begin{proposition}\label{prop:Nonlinearity}
Let $\delta \in (-2,1/2)$, and $g \in L^2_{\delta +2}$. Then the linear operator $F: W^{2,2}_{\delta +2} \times M^{2,2}_{\delta} \times \mathbb{R}^2 \rightarrow L^2_{\delta +2} \times L^2_{\delta+2}$ defined by \eqref{amplitude} and \eqref{phase}, is well defined and smooth.
\end{proposition}

\begin{proof}
Since $\delta \in (-2,1/2)$ the results form Lemma \ref{Holder}, and the embedding $W^{2,2}_{\delta+2} \subset BC(\mathbb{R}^2)$ suggest that all terms which do not involve the parameter $c$ are in the space $L^2_{\delta+2}$. Since all derivatives of $ \dfrac{\chi(|x|)}{|x|} $ are bounded, the only terms we need to worry about come from the expression $| \nabla \phi|^2$.
Recall here that $\phi = \tilde{\phi} + c P$, with $P= \dfrac{\chi(|x|)}{|x|}$ and $\tilde{\phi} \in M^{2,2}_{\delta}$, so that
\[ | \nabla \phi|^2 = |\nabla \tilde{\phi} |^2 + 2 c \nabla \tilde{\phi} \cdot \nabla P + c^2 | \nabla P |^2.\]
It is clear from Lemma \ref{Holder} that the expression $ |\nabla \tilde{\phi} |^2\in L^2_{\delta+2}$. Also, because $\nabla P $ is bounded in compact sets and behaves like $\langle x \rangle^{-2}$ for large $|x|$,  a straightforward calculation shows that $\nabla \tilde{\phi} \cdot \nabla P$ is in the desired space.
Finally, since $\delta  < 1/2$ the following integral converges
\[ \int_{\mathbb{R}^3} | \nabla P |^{4} \langle  x \rangle^{2(\delta +2)} \;dx \leq \int_{1}^{\infty} r^{2(\delta+2) -8} r^2 \;dr. \]
Given that all non-linear terms are defined via superposition operators of algebraic functions, they are smooth once well defined. This completes the proof.

\end{proof}

\section{Numerical Results}\label{numerics}
For the numerical simulations we consider the perturbed complex Ginzburg-Landau equation in a co-rotating frame,
\begin{equation}\label{numCGL}
A_t = ( 1+ i) \Delta A +( 1 +5 i ) A - (1+5i)A|A|^2 + i \ep g(x)A.
\end{equation}
The initial condition is the steady state $A=1$, and we take $\ep = 0.5$ and define the inhomogeneity as, 
\begin{equation}\label{inhom}
g(x,y,z) =(1 + x^2 +y^2 +z^2)^{-\alpha}.
\end{equation}
The domain is a cube of length $l = 40$ and the results are taken at time $T=500$ for different values of $\alpha$. Each value of $\alpha$ corresponds to a region in $\delta-$space for which the linearization $\hat{L}$ has different Fredholm properties (see Table \ref{deltaTable}). All numerical simulations were done on Matlab using  exponential time difference combined with an order four Runge-Kutta method. The grid size used was $N= 256$ and time step $h=1$. For more details on the code see \cite{K03, KT05}.

{\footnotesize
\begin{table}[h]
\centering{
\begin{tabular}{|c |c | c | c | c | c | c | c | c | c | }
\hline
Operator $\hat{L}$ is & \multicolumn{4}{|c|}{Invertible} &\multicolumn{3}{|c|}{Fredholm index -1}&\multicolumn{2}{|c|}{Fredholm index -3}\\
\hline
$\delta$-range & \multicolumn{4}{|c|}{$-3/2 < \delta < -1/2 $} & \multicolumn{3}{|c|}{ $ -1/2 < \delta <1/2$} & \multicolumn{2}{|c|}{$1/2< \delta< \infty$}\\
\hline
$\alpha$ & 1.2&1.3&1.4 &1.5&1.6 &1.8&2&2.2&2.4\\
\hline
$m_{\phi}$ &-0.608 &0.736 & -0.708 &-0.806& -0.949 & -1.029 & -1.066 & -1.046 & -1.06\\
\hline
\end{tabular}
}
\caption{The inhomogeneity, $g$, is in  $ L^2_{\delta+2}$ if $2\alpha > 3/2 +(\delta +2)$. The constant $m_{\phi}$ represents the decay rates for the phase ($\Phi(x,\ep,t) \sim |x|^{m_{\phi}}$) found in the numerical simulations (see figures at the end of Section \ref{numerics}). }
\label{deltaTable}
\end{table}
}
Table \ref{deltaTable} illustrates for which values of $\delta$ our results are valid. We are not interested in inhomogeneities with $\alpha <1$, since in this case our solutions blow up.  For inhomogeneities with $1<\alpha \leq 1.5$ we can pick $\delta \in (-3/2,-1/2)$. The result is that the linearization $L$ is invertible and in this case we do not have far field corrections. Consequently, we cannot make predictions on the asymptotic decay of the phase, but we can say that the phase $\phi$, viewed as a function of space alone, should satisfy the same properties as a function in $M^{2,2}_{\delta}$, i.e. $|\phi(x)|< C \langle x \rangle^{-\delta-3/2}$ (see Lemma \ref{decay}). On the other hand, for inhomogeneities with $\alpha >1.4$, the numerical results confirm that the phase decays at order $\rmO(1/|x|)$. 

We conclude this short section with some plots (Figures \ref{phase1} and \ref{amp1}) that illustrate the results of Table \ref{deltaTable}. Figure \ref{phase1} depict the phase of solutions to \eqref{numCGL} at the cross section $z=y=0$ and for different values of $\alpha$, and Figure \ref{amp1} depicts the amplitude of solutions for these same values. Notice that this las figure shows that the bounds for the amplitude in Theorem \ref{Th:main} are satisfied though not sharp, so that finding an asymptotic expansion for this quantity is still an open problem.
 
\begin{figure}[h]
\centering
\includegraphics[scale=0.6, trim= 0 0 0 0]{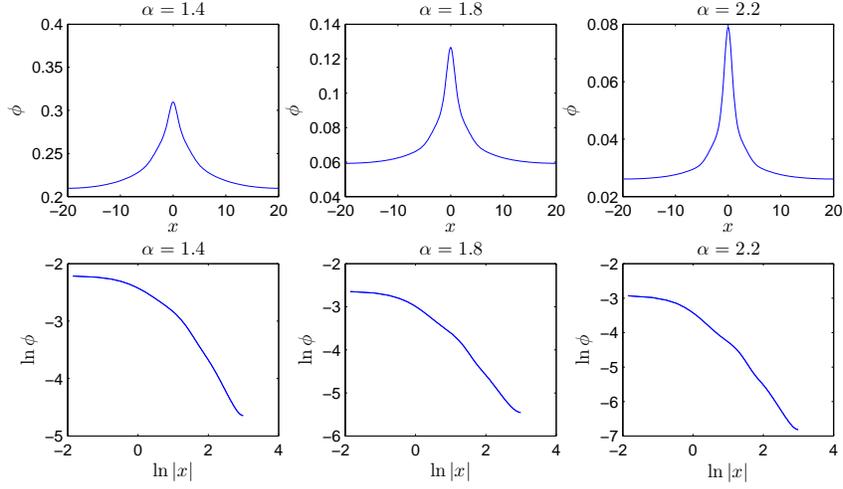} 
\caption{Plot of $\phi$ vs. $x$ and $\ln{\phi}$ vs. $\ln |x|$ at the cross section $z=0,y=0$ for values of $\alpha= 1.4, 1.8$ and $\alpha= 2.2$}
\label{phase1}
\end{figure}
\begin{figure}[h]
\centering
\includegraphics[scale=0.6, trim= 0 0 0 0]{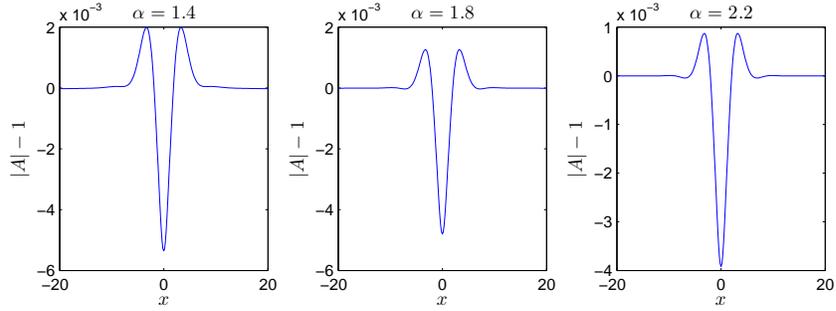} 
\caption{Plot of $|A|-1$ vs. $x$ at the cross section $z=0,y=0$ for values of $\alpha= 1.4, 1.8$ and $\alpha= 2.2$}
\label{amp1}
\end{figure}

\clearpage


\medskip
\medskip

\end{document}